\documentclass[10pt,legno]{article}
\usepackage{amsmath, amssymb}
\usepackage{latexsym}
\begin{document}
\newcommand{\qed}{\hfill \ensuremath{\square}}
\newtheorem{thm}{Theorem}[section]
\newtheorem{cor}[thm]{Corollary}
\newtheorem{lem}[thm]{Lemma}
\newtheorem{prop}[thm]{Proposition}
\newtheorem{defn}[thm]{Definition}
\newcommand{\proof}{\vspace{1ex}\noindent{\em Proof}. \ }
\newtheorem{pro}[thm]{proof}
\newtheorem{ide}[thm]{Idee}
\newtheorem{rem}[thm]{Remark}
\newtheorem{ex}[thm]{Example}
\bibliographystyle{plain}
\numberwithin{equation}{section}
%-------------------------------------------------------------------------
\numberwithin{equation}{section}
\newcounter{saveeqn}
\newcommand{\subeqn}{\setcounter{saveeqn}{\value{equation}}%
 \stepcounter{saveeqn}\setcounter{equation}{0}% =
\renewcommand{\theequation}{\mbox{\arabic{section}.\arabic{saveeqn}\alph{=
equation}}}} %\alph, \roman

\newcommand{\reseteqn}{\setcounter{equation}{\value{saveeqn}}%
\renewcommand{\theequation}{\arabic{section}.\arabic{equation}}}

% MATH =-------------------------------------------------------------------
\def\nm{\noalign{\medskip}}
\newcommand{\Om}{\Omega}
\newcommand{\om}{\omega}
\newcommand{\Real}{\mathbb{R}}
\newcommand{\nuu}{\tilde{\nu}}
\newcommand{\bohm}{{\partial}{\ohm}}
\newcommand{\la}{\langle}
\newcommand{\ra}{\rangle}
\newcommand{\ms}{\mathcal{S}_\ohm}
\newcommand{\mk}{\mathcal{K}_\ohm}
\newcommand{\mks}{\mathcal{K}_\ohm ^{\ast}}
\newcommand{\grad}{\bigtriangledown}
\newcommand{\ds}{\displaystyle}
\newcommand{\pf}{\medskip \noindent {\sl Proof}. ~ }
\newcommand{\p}{\partial}
\renewcommand{\a}{\alpha}
\newcommand{\z}{\zeta}
\newcommand\q{\quad}
\newcommand{\pd}[2]{\frac {\p #1}{\p #2}}
\newcommand{\pdl}[2]{\frac {\p^2 #1}{\p #2}}
\newcommand{\dbar}{\overline \p}
\newcommand{\eqnref}[1]{(\ref {#1})}
\newcommand{\na}{\nabla}
\newcommand{\ep}{\epsilon}
\newcommand{\vp}{\varphi}
\newcommand{\fo}{\forall}
\newcommand{\Scal}{\mathcal{S}}
\newcommand{\Dcal}{\mathcal{D}}
\newcommand{\Kcal}{\mathcal{K}}
\newcommand{\Acal}{\mathcal{A}}
\newcommand{\Ecal}{\mathcal{E}}
\newcommand{\Ncal}{\mathcal{N}}
\newcommand{\Abar}{\overline A}
\newcommand{\Rcal}{\mathcal{R}}
\newcommand{\Lcal}{\mathcal {L}}
\newcommand{\Gcal}{\mathcal {G}}
\newcommand{\Cbar}{\overline C}
\newcommand{\Ebar}{\overline E}
\newcommand{\RR}{\mathbb{R}}
\newcommand{\CC}{\mathbb{C}}
\newcommand{\NN}{\mathbb{N}}
\newcommand{\Z}{\mathbb{Z}}

\title{ Reconstructing  Small Perturbations of an  Obstacle
 for Acoustic Waves from Boundary Measurements  on the Perturbed Shape Itself}

\date{}

\author{Habib Zribi (zribi.habib@yahoo.fr)\thanks{Departement of Mathematics, College of Science,
University of Hafr Al Batin, P.O. 1803, Hafr Al Batin 31991, Saudi Arabia.}}

\maketitle

\begin{abstract}
We derive relationships between the shape deformation  of an impenetrable obstacle and  boundary measurements
of scattering fields on the perturbed shape itself.
Our derivation is rigourous by using systematic way, based  on layer potential techniques and the field expansion (FE) method (formal derivation).
We extend these techniques to derive asymptotic expansions  of the  Dirichlet-to-Neumann (DNO) and Neumann-to-Dirichlet (NDO) operators in terms of the small perturbations of the obstacle  as well as  relationships between the shape deformation  of an  obstacle and boundary measurements
of DNO or  NDO  on the perturbed shape itself. All  relationships lead us to very effective algorithms for determining lower-order Fourier coefficients of the shape perturbation of the obstacle.

\end{abstract}

\noindent {\footnotesize Mathematics Subject Classification
(MSC2000): 35B30, 35R30, 35C20, 35B40}


%%%%%%%%%%%%%%%%%%%%%%%%%%%%%%%%%%%%%%%%%%%%%%%%%%%%%%%%%%
\section{Introduction and statements of main results}
 Let us consider a situation, where we have an incident  wave $u^{in}$ propagating in a homogeneous isotropic medium $\RR^n$ for $n=2$ or $3$,  contains a bounded scatterer $D$ with
$C^{2}$-boundary,  which is either a sound-soft or a sound-hard impenetrable
obstacle. The wave  will scatter by the obstacle and we can express
the total wave  field around  the object as the sum of $u^{in}$ and a scattered wave $u^s$. The behavior of the scattered wave  will depend on both the incident wave and the shape  and the physical properties of the object.
The most inverse shape problems are to determine the shape of an object from measurements of scattered waves. Since the scattering field $u^s$ satisfies
\begin{equation}\label{u}
\left\{
  \begin{array}{ll}
\ds \Delta u^{s}+k^2 u^{s} =0 \quad \mbox{ in } \RR^n \backslash \overline{D}, \\
\nm \ds u^{s}=-u^{in} \quad \big(\mbox{or } \pd{u^s}{\nu}=-\pd{u^{in}}{\nu}\big)\quad \mbox{on  }  \p D,\\
\nm \ds \Big|\frac{\p u^{s}}{\p |x|}-ik u^{s}\Big|=O\big(|x|^{-\frac{n+1}{2}}\big) \q\mbox{as } |x|\rightarrow \infty,
  \end{array}
\right.
\end{equation}
where the wave number  $k>0$ and $\nu$ is the unit outward normal to the domain $D$.
%which satisfies $\Delta u^i+k^2 u^i=0$ in $\RR^2\backslash \overline{D}$.

Let  $D_\ep$ be an $\ep$-perturbation of $D$,  {\it i.e.}, there is a
function $h\in C^{1}(\p D)$ such that $\p D_\ep$ is given by
\begin{align*}
\ds \partial {D_\ep}=\big \{\tilde{x}=x+\ep h(x)\nu(x):=\Psi_\ep (x) | x \in \p D
\big \}.
\end{align*}
Let $u_{\ep}^{s}$  be the scattered field by $D_\ep$  which satisfies
\begin{equation}\label{uep}
\left\{
  \begin{array}{ll}
\ds \Delta u_{\ep}^{s}+k^2 u_{\ep}^{s} =0 \quad \mbox{in } \RR^n\backslash \overline{D_{\ep}}, \\
\nm \ds
u_{\ep}^{s}=-u_{i} \q \big (\mbox{or } \pd{u_{\ep}^s}{\nu}=-\pd{u^i}{\nu}\big)\quad \mbox{on  } \p D_\ep,\\
\nm \ds \Big|\frac{\p u_{\ep}^{s}}{\p |x|}-ik u_{\ep}^{s}\Big|=O\big(|x|^{-\frac{n+1}{2}}\big) \q\mbox{as } |x|\rightarrow \infty.
  \end{array}
\right.
\end{equation}

In this work, we consider the inverse acoustic obstacle scattering problems involve reconstructing the shape perturbation  of an obstacle from measurements of scattered fields. These inverse scattering problems are considerably more difficult to solve because they are nonlinear and ill-posed: the solution has an unstable dependence on the input data.  We propose a  way to determine the shape perturbation of an obstacle $D$ from boundary measurements on the perturbed obstacle $ D_\ep$, we get   relationships between the shape deformation $h$ and measurements of $u_\ep^s$ and $\p{u^s_\ep}/{\p \nu}$ on $\p D_\ep$.
In connection with our work, we should mention \cite{LLZ} on the reconstructing
small perturbations of bounded  scatterers    from electric or acoustic
far-field   measurements and  \cite{CMM,BHY,ZZ} on the reconstructing
of  locally small perturbations of half plan  from  acoustic
far-field or near-field  measurements.
%, and  \cite{tol} on the reconstruction of small perturbations of an interface for the inverse conductivity %problems.

Let $(v,w) \in H^{1}(\p D_\ep)\times H^{1}(\p D) $, we define 
\begin{align}\label{definition1}
\ds[v, w, \Psi_\ep, D]:=&\int_{\p D}\pd{v}{ \nu}\circ \Psi_\ep(x)  w(x) d\sigma(x) - \int_{\p D}v\circ \Psi_\ep(x) \pd{w}{ \nu}(x) d\sigma(x) \nonumber\\
\nm \ds= &\int_{\p D}\pd{v}{ \nu}(\tilde x)  w(x) d\sigma(x) - \int_{\p D}v(\tilde x) \pd{w}{ \nu}(x) d\sigma(x).
\end{align}
We denote by  $v^s$ the solution of the following system
\begin{equation}\label{v}
\left\{
  \begin{array}{ll}
\ds \Delta v^{s}+k^2 v^{s} =0 \quad \mbox{in } \RR^n \backslash \overline{D}, \\
\nm  \ds \Big|\frac{\p v^{s}}{\p |x|}-ik v^{s}\Big|=O\big(|x|^{-\frac{n+1}{2}}\big)\q\mbox{as } |x|\rightarrow \infty.
  \end{array}
\right.
\end{equation}

The main results of this paper is the following theorem, a rigourous derivation of the leading order term in the asymptotic expansion of $[u_\ep^s,v^s,\Psi_\ep, D]$  as $\ep \rightarrow 0$, based on the FE method and layer potential techniques.
\begin{thm} \label{main-Theorem}  Let  $u^s$,  $u^s_\ep$, and $v^s$ be the solutions of \eqref{u}, \eqref{uep}, and \eqref{v}, respectively. For the case
 of a sound-soft obstacle, we suppose that $k^2$
is not an eigenvalue of $-\Delta$ on $D$ with Neumann boundary condition and $u^{in} \in \mathcal{C}^1(\p D)$, while for the case of a sound-hard obstacle, we suppose that $k^2$
is not an eigenvalue of $-\Delta$ on $D$ with Dirichlet boundary condition and $u^{in} \in \mathcal{C}^2(\p D)$.
 The following  asymptotic  expansions  hold:
\begin{align}
\ds [u_\ep^s, v^s, \Psi_\ep, D]=\ep \int_{\p D}h\bigg[\frac{\p u^{s}}{\p T}\frac{\p v^{s}}{\p T}+(n-1)\tau \pd{u^{s}}{\nu} v^{s}
-\pd{u^{s}}{\nu} \pd{v^{s}}{\nu}-k^2 u^s v^s\bigg]d\sigma+O(\ep^2),
\label{main-Theorem-equality}
\end{align}
where  $T$ is the tangential vector to $\p D$ and $\tau$  is the mean curvature of $D$. Here the remainder $O(\ep^2)$ depends  only
on the $\mathcal{C}^2$-norm of $X$, the $\mathcal{C}^1$-norm of $h$,  and $k$.
\end{thm}

The  term in the left-hand side of  \eqref{main-Theorem-equality}
 can be determined by measurements as following:
 $\p{v^s}/{\p \nu}(x_i)$ and $v^s(x_i)$ are computed on  some locations $\{x_i\}$
 on the boundary $\p D$ which are supposed to be present before small perturbations of the shape $ D$
 and therefore the perturbed locations under the deformation can be  used  to measure
 the  fields  $\p{u_\ep^s}/{ \p \nu}(\tilde {x_i})$ and $u_\ep^s(\tilde {x_i})$ on $\p D_\ep$.
Our asymptotic expansions are still valid in the case of small perturbations of a locally half plan
  ($\tau=0$) and an obstacle of a small volume ($\tau\sim1/\ep$), but more elaborate arguments are needed for proofs. We derive relationships similar to \eqref{main-Theorem-equality} between the shape deformation  of an  obstacle and boundary measurements of   DNO or  NDO  on the perturbed shape itself.

Assuming that the unknown  object boundary is a small perturbation of a circle or a ball. The relationships between the shape deformation  of an  obstacle and one of boundary measurements of scattered fields,    DNO, and   NDO  are used for determining
lower-order Fourier coefficients of the shape perturbation of the object.

These relationships  could be used  to develop  effective algorithms to determine certain properties of the shape
perturbation of an impenetrable obstacle based on boundary measurements on the perturbed shape itself and to design new tools  for solving shape optimization problems: the idea would be to compute the gradient of some target functional using our asymptotic expansions with respect to the shape of the object.
%based on measurements on the shape itself.
To do  this, we refer to asymptotic  formulae related to  measurements in the same sprit,
 generalized polarization tensors and modal measurements that have been obtained in the recent papers \cite{ABFKL,AKLZ2}.

%and more arguments are needed

In this paper, we mainly focus on the derivation of the theorem \ref{main-Theorem} in two dimensions by  systematic way,  based  on the FE method and layer potential techniques. We   prove Theorem \ref{main-Theorem} in three dimensions   by the FE method,  it can be done by layer potential techniques  in exactly the same manner as in two dimensional  case  by
using \cite{KZ2}. We extend these techniques to  derive asymptotic expansions  of  the DNO and NDO in terms of the small  perturbations of the object shape.

This paper is organized as follows: In section $2$, we formally derive the asymptotic expansions in \eqref{main-Theorem-equality}  by using the FE method (Theorem \ref{main-Theorem}). In section $3$, we review some definitions and preliminary results on the layer potentials for Helmholtz equation and  derive asymptotic expansions of layer potentials. In section $4$, based on layer potential techniques we prove that in fact the formal expansion holds in two dimensions (Theorem \ref{main-Theorem}). In section $5$, we rigourously derive asymptotic expansions  for the NDO and DNO   as well as  relationships between the shape deformation  $h$ and measurements
of   DNO or  NDO. In the last section,  we present  algorithms to determine
the shape deformation $h$.
\section{Formal derivations: FE method}
 The following lemma is of use to us. See,  for instance \cite{book-CK,book-Kirsch}.

\begin{lem}\label{Important-Lemma}  Let $v_j$  satisfy \eqref{v}
for  $j=1,2$. Then
\begin{align}\label{main-equality}
\ds \int_{\p D}\Big(\pd{v_1}{ \nu}  v_2- v_1  \pd{v_2}{ \nu} \Big)d\sigma=0.
\end{align}
\end{lem}

Let $u_\ep$ be the solution to \eqref{uep}. In order to derive a formal asymptotic expansion for $u_\ep$, we apply the FE method, see \cite{NR, LLZ, Z, KZ2}. Firstly, we expand $u_\ep$ in powers of $\ep$, ${i.e.}$
\begin{equation}\label{heatExp}
\ds u_\ep^s(x)=u_0(x)+\ep u_1(x)+\ep^2 u_2(x)+\cdots,\quad x\in
\RR^d\backslash \overline{D_\ep},
\end{equation}
where  $u_l$  satisfies
\begin{equation}\label{u-l}
\left\{
  \begin{array}{ll}
\ds  \Delta u_{l}+k^2 u_{l} =0 \quad \mbox{in } \RR^d \backslash \overline{D}, \\
\nm \ds
 \Big|\frac{\p u_l}{\p |x|}-ik u_l\Big|=O\big(|x|^{-\frac{d+1}{2}}\big)\q\mbox{as }|x|\rightarrow \infty.
  \end{array}
\right.
\end{equation}
From $u_\ep(\tilde x)=-u^{in}(\tilde x)$ $\big(\mbox{or }\p u_\ep/\p \nu(\tilde x)=-\p u^{in}/\p \nu(\tilde x)\big)$ for $ \tilde x \in \p D_\ep$, we get  $u_0(x)=-u^{in}(x)$   $\big(\mbox{or }\p u_0/\p \nu( x)=-\p u^{in}/\p \nu( x)\big)$ for $x \in \p D.$ Note that $u_0\equiv u^s.$

\textbf{{Two dimentional case}:} Let $a,b \in \RR$, with $a<b$ and let $X(t):[a,b]\rightarrow \RR^2$
be the arclength parametrization of $\p D$, namely,  $X$ is a $\mathcal{C}^2$-function
 satisfying $|X^{'}(t)|=1$ for all $t\in [a,b]$ and such that
\begin{align*}
\ds  \p D:=\{x=X(t), t\in [a,b]\},
\end{align*}
with  $X^{'}(t)=T(x)$ and $X^{''}(t)=\tau (x) \nu(x).$

By $\frac{d}{dt}$, we denote the tangential derivative in the direction of $T(x)$. Let $\phi(x)\in \mathcal{C}^2([a,b])$ for $x=X(\cdot)\in \p D$. We have
\begin{align*}
\ds \frac{d\phi}{dt}(x)=\pd{\phi}{T}(x), \q\q\q  \Big(\frac{d}{dt} \Big)^2\phi(x)=\frac{\p^2 \phi}{\p T^2}(x)+\tau(x) \pd{\phi}{\nu}(x).
\end{align*}
As  a consequence, the restriction of $\Delta +k^2$ in $\RR^2\backslash \p D$
to a neighbourhood of $\p D$ can be expressed as follows:
\begin{align}\label{Laplacian-local}
\ds \Delta+k^2=\frac{\p^2}{\p \nu^2}-\tau \pd{}{\nu}+\Big(\frac{d}{dt} \Big)^2+k^2 \q \mbox{on } \p D.
\end{align}
We will sometimes use $h(t)$ for $h(X(t))$ and $h^{'}(t)$ for the tangential derivative of $h(x)$. Then,
$\tilde x =\Tilde X(t)=X(t)+\ep h(t) \nu(x)$ is a parametrization of $\p D_\ep$.

Let $x\in \p D$, then $\tilde x=x+\ep h(x)\nu(x)\in \p D_\ep$. It was proved in \cite{AKLZ1} that
$\nu(\tilde x)=\nu(x)-\ep h^{'}(t) T(x)+O(\ep^2)$. Using the  Taylor expansion and \eqref{Laplacian-local}, we write
\begin{align}\label{Taylor-normaluep}
\ds \pd{u^s_\ep}{\nu}(\tilde x)&=\Big(\nabla u^s(x)+\ep h(x)\nabla^2 u^s(x)\nu(x)+\ep \nabla u_1( x)\Big)\cdot \Big(\nu(x)-\ep h^{'}(t) T(x)\Big)+O(\ep^2)\nonumber\\
\nm \ds  &= \pd{u^s}{\nu}(x)-\ep \frac{d}{dt}\Big(h(x)\frac{d u^s}{dt}(x)\Big)+\ep \tau(x) h(x) \pd{u^s}{\nu}(x)
+ \ep \pd{u_1}{\nu}(x)-\ep k^2 h(x) u^s(x)\nonumber\\
\nm \ds &\q +O(\ep^2),
\end{align}
and
\begin{align}\label{Taylor-uep}
u^s_\ep(\tilde x) = u^s(x)+\ep h(x) \pd{u^s}{\nu}(x)+ \ep u_1(x)+O(\ep^2).
\end{align}
It follows from \eqnref{Taylor-normaluep} that
\begin{align}\label{first-term}
\ds \int_{\p D}\pd{u^s_\ep}{\nu}(\tilde x) v^s(x)d\sigma(x)&=\int_{\p D}\pd{u^s}{\nu} v^sd\sigma+\ep\int_{\p D}\pd{u_1}{\nu}  v^sd\sigma\nonumber\\
\nm \ds &\q+\ep \int_{\p D}h\bigg( \frac{\p u^s}{\p T} \frac{\p v^s}{\p T}
+\tau\pd{u^s}{\nu}  v^s-k^2 u^s   v^s\bigg) d\sigma+O(\ep^2).
\end{align}
According to \eqref{Taylor-uep}, we have
\begin{align}\label{second-term}
\ds \int_{\p D}u^s_\ep(\tilde x)  \pd{v^s}{\nu}(x)d\sigma(x)=&\int_{\p D}u^s \pd{v^s}{\nu}d\sigma +\ep\int_{\p D}u_1  \pd{v^s}{\nu}d\sigma+\ep \int_{\p D}h\frac{\p u^s}{\p \nu}  \frac{\p v^s}{\p \nu} d\sigma+O(\ep^2).
\end{align}
Subtracting \eqref{second-term} from \eqnref{first-term} yields
\begin{align}
\ds [u^s_\ep,v^s,\Psi_\ep, D]=&\ep \int_{\p D}h \Big(\frac{\p u^{s}}{\p T}\frac{\p v^{s}}{\p T}+\tau \pd{u^{s}}{\nu} v^{s}
-\pd{u^{s}}{\nu} \pd{v^{s}}{\nu}-k^2 u^s v^s\Big)d\sigma\nonumber\\
\nm \ds &+\int_{\p D}\Big(\pd{u^s}{\nu} v^s-u^s \pd{v^s}{\nu}\Big)d\sigma+\ep\int_{\p D}\Big(\pd{u_1}{\nu}  v^s -u_1 \pd{v^s}{\nu}\Big)d\sigma+O(\ep^2).\label{11}
\end{align}
By Lemma  \ref{Important-Lemma}, the second and the third  integrals in the right-hand side of \eqnref{11} vanish. Thus Theorem \ref{main-Theorem}
is proved formally in two dimensions. For proof see Section \ref{proof-theorem1}.

\textbf{Three dimensional case:} Let $\vartheta$ be an open subset of $\RR^2$. Let $X(\varphi, \theta)$ be
an orthogonal parametrization  of  the surface $\p D$, that is, $$\p D:=\big\{x= X(\varphi, \theta), (\varphi, \theta)\in \vartheta\big\}$$ for $X \in \mathcal{C}^2(\vartheta)$,  where $\big(X_\varphi:=\frac{dX}{d\varphi}\big)\cdot \big(X_\theta:=\frac{dX}{d\theta}\big)=0$. The vectors  $T_\varphi=X_\varphi/|X_\varphi|$ and $T_\theta=X_\theta/|X_\theta|$ form an  orthonormal basis for the tangent plane to $\p D$ at $x=X(\varphi, \theta)$. The tangential derivative on $\p D$ is defined  by $\pd{}{T}=\frac{\p }{\p T_\varphi}T_{\varphi}+\pd{}{T_\theta}T_{\theta}$.

Let   $\mathcal{{G}}$ be  the matrix of the first fundamental form with respect to the basis  $\{X_\varphi, X_\theta\}$ which is given by
\begin{equation*}
\ds \mathcal{G} =
\begin{pmatrix}
|X_\varphi|^2 & 0 \\
0 & |X_\theta|^2
\end{pmatrix}.
\end{equation*}
For $v\in \mathcal{C}^2(\vartheta)$. The gradient operator in local coordinates  satisfies
\begin{align}\label{local-gradient}
\nabla_{\varphi, \theta}v=\sqrt{\mathcal{G}_{11}}\pd{v}{T_\varphi}T_\varphi+\sqrt{\mathcal{G}_{22}}\pd{v}{T_\theta}T_\theta,\q
\mathcal{G}^{-1} \nabla_{\varphi, \theta}v=\frac{1}{\sqrt{\mathcal{G}_{11}}}\pd{v}{T_\varphi}T_\varphi+\frac{1}{\sqrt{\mathcal{G}_{22}}}
\pd{v}{T_\theta}T_\theta,
\end{align}
and the  restriction of $\Delta +k^2$ in $\RR^3\backslash \p{D}$
to a neighbourhood of $\p D$ can be expressed as follows:
\begin{align*}
\Delta v+k^2 v=\frac{\p^2v}{\p \nu^2}-2\tau \pd{v}{\nu}+\frac{1}{\sqrt{det\mathcal{G}}}
\nabla_{\varphi, \theta}\cdot \Big(\sqrt{det\mathcal{G}} \mathcal{G}^{-1}\nabla_{\varphi, \theta} v\Big)+k^2 v \q \mbox{on } \p D.
\end{align*}

We use $h(\varphi, \theta)$ for simplifying the term $h(X(\varphi, \theta))$
and $h_\varphi(\varphi, \theta)$, $h_\theta(\varphi, \theta)$ for the tangential derivatives of $h(X(\varphi, \theta))$. Then,
$\tilde x =X(\varphi, \theta)+\ep h(\varphi, \theta) \nu(x)$ is a parametrization of $\p D_\ep$.
 It was proved in \cite{KZ2} that
$$\nu(\tilde x)=\nu(x)-\ep\Big( \frac{h_{\varphi}}{\sqrt{\mathcal{G}_{11}}} T_\varphi+ \frac{h_{\theta}}{\sqrt{\mathcal{G}_{22}}} T_\theta\Big)+O(\ep^2).$$
Let $\tilde x=x+\ep h(x) \nu(x) \in \p D_\ep$ for $x\in \p D$. The following Taylor expansions  hold
\begin{align}\label{Taylor-normaluep-3D}
\ds \pd{u^s_\ep}{\nu}(\tilde x)&=\Big(\nabla u^s(x)+\ep h(x)\nabla^2 u^s(x)\nu(x)+\ep \nabla u_1( x)\Big)\cdot \nu(\tilde x)+O(\ep^2)\nonumber\\
\nm \ds  &= \pd{u^s}{\nu}(x)+2\ep \tau(x) h(x) \pd{u^s}{\nu}(x)
+ \ep \pd{u_1}{\nu}(x)-\ep k^2 h(x) u^s(x)\nonumber\\
\nm \ds &\q -\frac{\ep}{\sqrt{det\mathcal{G}}}
\nabla_{\varphi, \theta}\cdot \Big(h (x)\sqrt{det\mathcal{G}} \mathcal{G}^{-1}\nabla_{\varphi, \theta} u^s(x)\Big)+O(\ep^2),
\end{align}
and
\begin{align}\label{Taylor-uep-3D}
u^s_\ep(\tilde x) = u^s(x)+\ep h(x) \pd{u^s}{\nu}(x)+ \ep u_1(x)+O(\ep^2).
\end{align}
Inserting  the two expansions in  \eqnref{Taylor-normaluep-3D} and \eqref{Taylor-uep-3D} into  \eqnref{definition1}, we obtain
\begin{align}
\ds [u^s_\ep,v^s,\Psi_\ep, D]=&\ep \int_{\p D}h \Big(2\tau \pd{u^{s}}{\nu} v^{s}
-\pd{u^{s}}{\nu} \pd{v^{s}}{\nu}-k^2 u^s v^s\Big)d\sigma\nonumber\\
\nm \ds &- \ep \int_{\p D} \frac{1}{\sqrt{det\mathcal{G}}}
\nabla_{\varphi, \theta}\cdot \Big(h \sqrt{det\mathcal{G}} \mathcal{G}^{-1}\nabla_{\varphi, \theta} u^s\Big) v^s d\sigma\nonumber\\
\nm \ds &+\int_{\p D}\Big(\pd{u^s}{\nu} v^s-u^s \pd{v^s}{\nu}\Big)d\sigma+\ep\int_{\p D}\Big(\pd{u_1}{\nu}  v^s -u_1 \pd{v^s}{\nu}\Big)d\sigma+O(\ep^2).\label{22}
\end{align}
According  to Lemma  \ref{Important-Lemma}, the fourth and the fifth  integrals  in the right-hand side of \eqnref{22} vanish. By integrating by parts and \eqref{local-gradient}, we find that 
\begin{align*}
\ds \int_{\p D} \frac{1}{\sqrt{det\mathcal{G}}}
\nabla_{\varphi, \theta}\cdot \Big(h \sqrt{det\mathcal{G}} \mathcal{G}^{-1}\nabla_{\varphi, \theta} u^s\Big) v^s d\sigma&=\int_{\vartheta}
\nabla_{\varphi, \theta}\cdot \Big(h \sqrt{det\mathcal{G}} \mathcal{G}^{-1}\nabla_{\varphi, \theta} u^s\Big) v^s d\varphi d\theta\\
\nm \ds &= - \int_{\vartheta}
h \sqrt{det\mathcal{G}} \mathcal{G}^{-1}\nabla_{\varphi, \theta} u^s \cdot \nabla_{\varphi, \theta}v^s d\varphi d\theta\\
\nm \ds &= - \int_{\p  D} h  \mathcal{G}^{-1}\nabla_{\varphi, \theta} u^s \cdot \nabla_{\varphi, \theta}v^s d\sigma\\
\nm \ds &= - \int_{\p  D} h  \Big(\pd{u^s}{T_\varphi} \pd{v^s}{T_\varphi}+\pd{u^s}{T_\theta} \pd{v^s}{T_\theta} \Big) d\sigma\\
\nm \ds &= - \int_{\p  D} h  \pd{u^s}{T}  \pd{v^s}{T} d\sigma.
\end{align*}
Thus Theorem \ref{main-Theorem}
is proved formally in three dimensions.

\section{Layer potentials for Helmholtz equation}

\subsection{Definitions and Preliminary results}
We start to review some basic facts in the theory of layer potentials.
Let $\Gamma_k(x)$ be the fundamental solution of $\Delta +k^2$ in $\RR^2$, that is  for $x\neq 0$,
\begin{align*}
\ds \Gamma_k(x)=-\frac{i}{4}H^1_0(k|x|),
\end{align*}
where $H^1_0$ is the Hankel  function of the first kind of order $0$. We have the following Taylor expansion  of $H^1_0(x)$ as $|x|\rightarrow 0$ \cite{book-Watson}:
\begin{align}\label{asymptotic-Hankel}
\ds -\frac{i}{4}H^1_0(k|x|)=\frac{1}{2\pi}\sum_{n=0}^{+\infty}(-1)^n \frac{k^{2n}}{2^{2n}(n!)^2}
|x|^{2n}\Big(\ln(|x|)+\ln(k \gamma)-\sum_{j=1}^{n}\frac{1}{j}\Big),
\end{align}
where $2\gamma=e^{\tilde \gamma -i\pi /2}$, and $\tilde \gamma$ is Euler's constant.

According to Leibniz's
rule, the $p$th  derivative of $r^{2n} \ln(r)$ is given by
\begin{align*}
\ds \big(r^{2n}\ln (r)\Big)^{(p)}=\sum_{l=0}^{p}C_{p}^l (r^{2n})^{(l)}\big(\ln(r)\big)^{(p-l)}=(r^{2n})^{(p)}\ln(r) +\sum_{l=0}^{p-1}C_{p}^l (r^{2n})^{(l)}\big(\frac{1}{r}\big)^{(p-l-1)},
\end{align*}
where $C_p^l$ is a binomial coefficient, and then, it follows from \eqref{asymptotic-Hankel} that
\begin{align*}
\ds -\frac{ik^p}{4}H_0^{1^{(p)}}(kr) r^p \mbox{ is continuous at zero  for } p\geq 1.
\end{align*}

For a bounded domain $D$ in $\RR^2$ and $k>0$ let $\Scal_{D}^k$ and $\Dcal_D^k$ be the single and double layer potentials defined by $\Gamma_k$, that is,
\begin{align*}
\ds \Scal_D^{k}[\phi](x)&=\int_{\p D}\Gamma_k(x-y)\phi (y) d\sigma(y), \q x\in \RR^2,\\
\nm\ds \Dcal_D^{k}[\phi](x)&=\int_{\p D} \pd{\Gamma_k(x-y)}{\nu(y)} \phi (y) d\sigma(y), \q x\in \RR^2\backslash \p D.
\end{align*}
It is well-known, see Theorem $3.1$ of \cite{book-CK}, that
\begin{align}
\ds \pd{\Scal_D^{k}[\phi]}{\nu}\Big |_{\pm}(x)&=\Big(\pm \frac{1}{2}I+ (\Kcal_D^{k})^*\Big)[\phi] (x) \q \mbox{a.e. } x\in \p D,\label{jump-normal-single-layer}\\
\nm\ds \Dcal_D^{k}[\phi]\Big |_{\pm}(x)&=\Big(\mp \frac{1}{2}I+ \Kcal_D^{k}\Big)[\phi] (x) \q \mbox{a.e. } x\in \p D, \label{jump-double-layer}
\end{align}
for $\phi \in L^2(\p D)$, where $\Kcal_D^{k}$ is the operator on $L^2(\p D)$  defined by
\begin{align*}
\ds \Kcal_D^{k}[\phi] (x)&=\mbox{p.v.} \int_{\p D}\pd{\Gamma_k(x-y)}{\nu(y)} \phi (y) d\sigma(y),
\end{align*}
and $ (\Kcal_D^{k})^*$ is the $L^2$-adjoint of $\Kcal_D^k$. Here p.v. denotes the cauchy principal value. The operator $\Kcal_D^k$ is known to be bounded on $L^2(\p D)$ \cite{CMM}.

If $D$ has a $\mathcal{C}^2$ boundary and $\phi \in H^{\frac{1}{2}}(\p D)$, then $\p (\Dcal_D^{k}[\phi])/\p \nu$ does not have a jump across $\p D$, that is,
\begin{align}\label{jump-normal-double-layer-potential}
\ds \pd{\Dcal_D^{k}[\phi]}{\nu}\Big|_{+}(x)=\pd{\Dcal_D^{k}[\phi]}{\nu}\Big|_{-}(x), \q x\in \p D.
\end{align}
Recall that the operators $\ds \frac{\p^2\Dcal^k_D[\phi]}{\p \nu^2}$,  $\ds \Big(\frac{d}{dt}\Big)^2\Dcal^k_D[\phi]$, and
$\Dcal^k_D[\phi]$ are not  continuous on $\p D$, but it follows from $(\Delta +k^2)\Dcal^k_D[\phi]=0$ in $\RR^2\backslash \p D$ and \eqref{Laplacian-local} that $\ds \frac{\p^2\Dcal^k_D[\phi]}{\p \nu^2}+ \Big(\frac{d}{dt}\Big)^2\Dcal^k_D[\phi]+k^2 \Dcal^k_D[\phi]$ is continuous on $\p D$ and we have
\begin{align}\label{jum-double-nu-squared}
\ds \frac{\p^2\Kcal^k_D[\phi]}{\p \nu^2}+ \Big(\frac{d}{dt}\Big)^2\Kcal^k_D[\phi]+k^2 \Kcal^k_D[\phi]=\tau  \pd{\Dcal^k_D[\phi]}{\nu}\q \mbox{on } \p D.
\end{align}
If $\phi \in \mathcal{C}^2(\p D)$, then we get from  \eqref{jump-double-layer} and \eqref{jum-double-nu-squared}  that
\begin{align*}
 \ds \frac{\p^2\Dcal^k_D[\phi]}{\p \nu^2}\Big|_{\pm}=\pm  \frac{1}{2}\Big(\frac{d}{dt}\Big)^2\phi
 \pm \frac{k^2}{2}\phi+\pdl{\Kcal^k_D[\phi]}{\nu^2}\q \mbox{on } \p D.
\end{align*}

The following  uniqueness result for the exterior Helmholtz problem holds. (See  \cite{book-CK,book-AFKRYZ}).
\begin{lem} \label{uniqness-lemma} Let $D$ be a bounded Lipschitz domain in $\RR^2$. Let $w \in H^{1}_{loc}(\RR^2\backslash \overline{D})$ satisfy
 $$
 \left\{ \begin{array}{l}
\ds \Delta w+k^2 w=0 \quad\mbox{in }\RR^2\backslash \overline{D},\\
\nm\ds\bigg|\pd{w}{r}-ik w \bigg|=O\Big(1/r^{\frac{3}{2}}\Big)\q \mbox{as } r=|x|\rightarrow +\infty \q \mbox{uniformy in }\frac{x}{|x|},\\
\nm\ds w=0 \mbox{ or  } \pd{w}{\nu}=0\quad\mbox{on } \p D.\end{array}\right.
 $$
 Then, $w\equiv0$ in $\RR^2\backslash \overline{D}$.
\end{lem}
The following lemma is important for us.
\begin{lem} \label{lemma-invertible} The following properties hold:
 \begin{enumerate}
     \item  Suppose that $D$ is of class $\mathcal{C}^{2}$. If $k^2$
     is not an eigenvalue of $-\Delta$ on $D$ with Dirichlet boundary condition, then the operator $(1/2)I+\Kcal_{D}^k: L^2(\p D) \rightarrow L^2(\p D)$ is invertible.
     \item   Let $D$ be a bounded Lipschitz domain. If $k^2$
     is not an eigenvalue of $-\Delta$ on $D$ with Neumann boundary condition, then the operator $-(1/2)I+\big(\Kcal_{D}^k\big)^*: L^2(\p D) \rightarrow L^2(\p D)$ is invertible.
   \end{enumerate}
\end{lem}
\begin{proof} The operators $\Kcal_{D}^k$ and  $\big(\Kcal_{D}^k\big)^*$ are compact.
Therefore, we can apply the Reisz-Fredholm theory.
Let $\phi \in L^2(\p D)$ such that $\big((1/2)I+\Kcal_{D}^k\big)[\phi]=0$.
Then  $v(x):=\Dcal_{D}^k[\phi]$ on $D$ is a solution to $\Delta v +k^2 v=0$ with the boundary
condition $v|_{-}=0$ on $\p D$. If $k^2$
     is not an eigenvalue of $-\Delta$ on $D$ with Dirichlet boundary condition, then $\Dcal_{D}^k[\phi]=0$
      in $D$. Since $\p (\Dcal_{D}^k[\phi])/\p \nu$ exists and has no jump across $\p D$, we get
$$
\pd{\Dcal_{D}^{k}[\phi]}{\nu}\bigg|_{+}=\pd{\Dcal_{D}^{k}[\phi]}{\nu}\bigg|_{-}=0 \q \mbox{on }\p D.
$$
One easily checks   that $v$  is a solution to $\Delta v +k^2 v=0$ on $\RR^2\backslash \overline{D}$
with the boundary
condition ${\p v}/{\p \nu}\big|_{+}=0$ on $\p D$ and satisfies the radiation condition. The uniqueness
result in Lemma \ref{uniqness-lemma} implies that  $\Dcal_{D}^k[\phi]=0$
in $\RR^2\backslash \overline{D}$. Therefore, we conclude
$$
 \phi=\Dcal_{D}^k[\phi]\big |_{-}- \Dcal_{D}^k[\phi]\big |_{+}=0.
$$
Suppose now $\big(-(1/2)I+\big(\Kcal_{D}^k\big)^*\big)[\psi]=0$. Define $w:=\Scal_{D}^k[\psi]$ on $\RR^2 \backslash \p D$. Therefore, $w$ is the solution to  $\Delta w+k^2 w=0$ in $D$ with the boundary
condition $\p{w}/{\p \nu}|_{-}=0$ on $\p D$. If $k^2$
     is not an eigenvalue of $-\Delta$ on $D$ with Neumann boundary condition, then $\Scal_{D}^k[\psi]=0$
      in $D$. Furthermore, $w$ is continuous in $\RR^2$, thus $w$ is a solution
  to $\Delta w+k^2 w=0$ on $\RR^2\backslash \overline{D}$
with the boundary
condition $w|_{+}=0$ on $\p D$ and satisfies the radiation condition. The uniqueness
result of Lemma \ref{uniqness-lemma} yields that  $\Scal_{D}^k[\psi]=0$
in $\RR^2\backslash \overline{D}$ and hence
$$
 \psi=\pd{\Scal_{D}^k[\psi]}{\nu}\Big |_{+}- \pd{\Scal_{D}^k[\psi]}{\nu}\Big |_{-}=0.
$$
\end{proof}

\subsection{Asymptotic of layer potentials}
Let $\tilde x, \tilde y \in \p D_\ep$, that is,
\begin{align*}
\ds \tilde x =x+\ep h(x)\nu(x),\q \q \q \tilde y =y+\ep h(y)\nu(y),
\end{align*}
for $x=X(t), y=X(s)\in \p D$. By, $\nu(\tilde y)$ and  $d\sigma(\tilde y)$ we denote the unit outward unit normal  and the length element to $\p D_\ep$ at $\tilde y$, respectively.   It was proved in \cite{AKLZ1} that
\begin{align}\label{n-tilde}
\ds \nu(\tilde y)=\frac{\nu(y)-\ep\Big(h(y)\tau(y)\nu(y)+h^{'}(s)T(y)\Big)}{\sqrt{\Big(1-\ep h(y) \tau(y)\Big)^2+\ep^2 \big(h^{'}(s)\big)^2}},
\end{align}
and
\begin{align}\label{sigma-tilde}
\ds d\sigma(\tilde y)= \sqrt{\Big(1-\ep h(y) \tau(y)\Big)^2+\ep^2 \big(h^{'}(s)\big)^2} d\sigma(y).
\end{align}
Since,
\begin{align}\label{tilde-x-y}
\ds \tilde y -\tilde x =y-x+\ep \Big(h(y)\nu(y)-h(x)\nu(x)\Big),
\end{align}
which yields
\begin{align}\label{x-y-square}
\ds |\tilde y -\tilde x |^2=|y-x|^2 \bigg(1+2\ep \frac{\la y-x, h(y)\nu(y)-h(x)\nu(x) \ra}{|y-x|^2}+\ep^2 \frac{\big|h(y)\nu(y)-h(x)\nu(x)\big|^2}{|y-x|^2}\bigg),
\end{align}
and hence
\begin{align}\label{1-over-x-y}
\ds \frac{1}{|\tilde y -\tilde x |^2}=\frac{1}{|y-x|^2}\cdot \frac{1}{{1+2\ep F(x,y)+\ep^2 G(x,y)}},
\end{align}
where
\begin{align*}
\ds F(x,y)=\frac{\la y-x, h(y)\nu(y)-h(x)\nu(x) \ra}{|y-x|^2}, \q\q G(x,y)=\frac{\big|h(y)\nu(y)-h(x)\nu(x)\big|^2}{|y-x|^2}.
\end{align*}
One can easily see that
\begin{align*}
\ds |F(x,y)|+|G(x,y)|^{\frac{1}{2}}\leq C \|X\|_{\mathcal{C}^2(\p D)} \|h\|_{\mathcal{C}^1(\p D)} \q \mbox{for all }
x,y \in \p D.
\end{align*}

In order to prove the asymptotic expansion of the operator $\Kcal_{D_\ep}^k$, we  investigate
\begin{align*}
\Big(k {H_0^{1}}^{'}(k|\tilde y -\tilde x|)|\tilde y -\tilde x|\Big)\frac{\la \tilde y-\tilde x, \nu(\tilde y)\ra}{|\tilde y-\tilde x|^2}d\sigma (\tilde y).
\end{align*}
By using \eqref{x-y-square}, we write
\begin{align}\label{H}
k {H_0^{1}}^{'}(k|\tilde y -\tilde x|)|\tilde y -\tilde x|=\sum_{n=0}^{\infty}\ep^n \mathbb{H}_n(x,y),
\end{align}
where the series converges absolutely and uniformly. In particular,
\begin{align*}
\mathbb{H}_0(x,y)=k {H_0^{1}}^{'}(k| y - x|)| y - x|,
\end{align*}
and
\begin{align*}
\mathbb{H}_1(x,y)=\Big[k^2 {H_0^{1}}^{''}(k| y - x|)| y - x|+k {H_0^{1}}^{'}(k| y - x|)\Big]\frac{\la y-x, h(y)\nu(y)-h(x)\nu(x)\ra}{|y-x|}.
\end{align*}
It follows from \eqref{n-tilde}, \eqref{sigma-tilde}, \eqref{tilde-x-y}, and  \eqref{1-over-x-y} that
{
\begin{align}\label{M}
\ds \frac{\la \tilde y-\tilde x, \nu(\tilde y)\ra}{|\tilde y-\tilde x|^2}d\sigma (\tilde y)&=
\frac{\la y-x+\ep\big(h(y)\nu(y)-h(x)\nu(x)\big), \nu(y)- \ep \big[h(y) \tau(y) \nu(y)+h^{'}(s)T(y)\big]\ra}{|y-x|^2}\nonumber\\
\nm \ds & \q \times \frac{1}{{1+2\ep F(x,y)+\ep^2 G(x,y)}} d\sigma(y)\nonumber\\
\nm \ds &:=\sum_{n=0}^{\infty}\ep^n \mathbb{M}_n(x,y)d\sigma(y),
\end{align}}
where the series converges absolutely  and uniformly. In particular, one can easily see that
$$
\mathbb{M}_0(x,y)=\frac{\la y-x, \nu(y)\ra}{|y-x|^2},
$$
and
\begin{align*}
 \ds \mathbb{M}_1(x,y)&= h(x)\bigg(-\frac{\la \nu(x), \nu(y)\ra}{|y-x|^2}+2 \frac{\la y-x, \nu(y)\ra \la y-x, \nu(x)\ra}{|y-x|^4}\bigg)\\
\nm \ds & \q + \bigg(\frac{h(y)}{|y-x|^2}-2 h(y)\frac{\big(\la y-x, \nu(y)\ra\big)^2}{|y-x|^4}- \frac{\la y-x, h(y) \tau(y)\nu(y)+h^{'}(s) T(y)\ra}{|y-x|^2} \bigg).
\end{align*}
Thus  we obtain  from \eqref{H}  and \eqref{M} that
$$
\Big(k H^{'}_0(k|\tilde y -\tilde x|)|\tilde y -\tilde x|\Big)\frac{\la \tilde y-\tilde x, \nu(\tilde y)\ra}{|\tilde y-\tilde x|^2}d\sigma (\tilde y)=\sum_{n=0}^{\infty}\ep^n \underbrace{\sum_{m=0}^{n} \mathbb{M}_m(x,y)
 \mathbb{H}_{n-m}(x,y)}_{:=\mathbb{K}_{n}(x,y)}d\sigma(y),
$$
with
\begin{align*}
-\frac{i}{4}\mathbb{K}_{0}(x,y)= -\frac{ik}{4} {H_0^1}^{'}(k| y - x|)   \frac{\la y-x, \nu(y)\ra}{|y-x|}=  \pd{\Gamma_{k}(x-y)}{\nu(y)},
\end{align*}
and
\begin{align*}
\ds -\frac{i}{4}\mathbb{K}_{1}(x,y)&= h(x)\bigg[\frac{ik^2}{4}{H_0^1}^{''}(k| y - x|)\frac{\la y-x, \nu(x)\ra \la y-x, \nu(y)\ra }{|y-x|^2}\\
\nm \ds &\q\q\q\q\q+\frac{ik}{4} {H_0^1}^{'}(k| y - x|)\bigg( \frac{\la \nu(x), \nu(y)\ra}{|y-x|}-\frac{\la y-x, \nu(x)\ra \la y-x, \nu(y)\ra }{|y-x|^3}\bigg) \bigg]\\
\nm \ds &\q + h(y)\bigg[-\frac{ik^2}{4}{H_0^1}^{''}(k| y - x|)\frac{\big(\la y-x, \nu(y)\ra\big)^2 }{|y-x|^2}\\
\nm \ds &\q\q\q\q\q\q\q\q\q\q\q-\frac{ik}{4} {H_0^1}^{'}(k| y - x|)\bigg( \frac{1}{|y-x|}-\frac{\big(\la y-x, \nu(y)\ra\big)^2 }{|y-x|^3}\bigg) \bigg]\\
\nm \ds &\q + \frac{ik}{4}{H_0^1}^{'}(k| y - x|)\frac{\la y-x, h(y)\tau(y)\nu(y)+h^{'}(s)T(y)\ra }{|y-x|^2}.
\end{align*}
Note that
\begin{align*}
\ds -\frac{i}{4}\mathbb{K}_{1}(x,y)&= h(x)\frac{\p^2 \Gamma_k(x-y)}{\p \nu(x)\p \nu(y)}+h(y)\frac{\p^2 \Gamma_k(x-y)}{\p \nu(y)^2}\\
\nm \ds &\q
-\tau (y)h(y)  \pd{\Gamma_{k}(x-y)}{\nu(y)}-h^{'}(s)\pd{\Gamma_{k}(x-y)}{T(y)}\\
\nm \ds &=h(x)\frac{\p^2 \Gamma_k(x-y)}{\p \nu(x)\p \nu(y)}-\frac{d}{ds}\Big(h(y)\frac{d \Gamma_k(x-y)}{ds} \Big)-k^2 h(y) \Gamma_k(x-y).
\end{align*}
In order to justify the last equality, we use $(\Delta +k^2)\Gamma_k(x-y)=0$ for $x\neq y$ and the representation of $\Delta +k^2$ on $\p D$  given in \eqref{Laplacian-local}.

Introduce a sequence of integral operators $\big(\Dcal_{D,n}^k\big)_{n\in \NN}$,  defined for any $\phi \in L^2(\p D)$ by
$$
\Dcal_{D,n}^k [\phi](x)=-\frac{i}{4} \int_{\p D}\mathbb{K}_{n}(x,y)\phi (y) d\sigma(y)\q \mbox{for } n\geq 0,
$$
where $\Dcal_{D,0}^k=\Kcal_{D}^k$ and  for $\phi \in \mathcal{C}^2(\p D)$ we have
\begin{align}\label{D-1}
\ds \Dcal_{D,1}^k [\phi](x)=-k^2\Scal_{D}^{k}[ h \phi](x)+ h(x) \pd{\Dcal_{D}^{k}[\phi]}{\nu}(x)-\Scal^k_{D}\bigg(\frac{d}{ds}\Big(h \frac{d \phi}{ds}\Big)\bigg)(x), \q x \in \p D.
\end{align}
It is easy to prove that the operator $\Dcal_{D,n}^k$ for $n\geq 1$ with the kernel $\mathbb{K}_n(x,y)$ is bounded in $L^2(\p D)$. In fact,  it is an immediate consequence of the celebrate theorem of Coifman-MacIntosh-Meyer, see \cite{CMM}.

in order to establish the asymptotic expansion of the operator $\p (\Dcal_{D_\ep})/\p \nu$  on  $\p D_\ep$, we next  investigate the following terms
$$
\Big[k {H_0^1}^{'}(k |\tilde x-\tilde y|)|\tilde x-\tilde y|\Big]\frac{\la \nu(\tilde x),\nu(\tilde y)\ra}{|\tilde x-\tilde y|^2}d\sigma(\tilde y),
$$
and
\begin{align*}
\Big[-k^2 {H_0^1}^{''}(k |\tilde x-\tilde y|)|\tilde x-\tilde y|^2+k {H_0^1}^{'}(k |\tilde x-\tilde y|)|\tilde x-\tilde y|\Big]\frac{\la \tilde y -\tilde x , \nu(\tilde x)\ra }{|\tilde x-\tilde y|^2}\frac{\la \tilde y -\tilde x , \nu(\tilde y)\ra }{|\tilde x-\tilde y|^2}
d\sigma(\tilde y).
\end{align*}
It follows from \eqref{n-tilde}, \eqref{sigma-tilde}, and \eqref{1-over-x-y} that
\begin{align}\label{L}
\ds& \frac{\la \nu(\tilde x),\nu(\tilde y)\ra}{|\tilde x-\tilde y|^2}d\sigma(\tilde y)\nonumber\\
\nm \ds&\q=
\frac{\Big[\nu(x)-\ep \Big(h(x)\tau(x)\nu(x)+h^{'}(t) T(x)\Big)\Big]\Big[\nu(y)-\ep \Big(h(y)\tau(y)\nu(y)+h^{'}(s) T(y)\Big)\Big]}{| x- y|^2}\nonumber\\
\nm \ds &\q\q \times \frac{1}{{1+2\ep F(x,y)+\ep^2 G(x,y)}}\frac{1}{\sqrt{\big(1-\ep h(x) \tau(x)\big)^2+\ep^2 \big(h^{'}(t)\big)^2}}d\sigma(y)\nonumber\\
\nm \ds&\q :=\sum_{n=0}^{\infty}\ep^n\mathbb{L}_n(x,y)d\sigma(y),
\end{align}
with
\begin{align*}
\ds \mathbb{L}_0(x,y)&=\frac{\la \nu( x),\nu(y)\ra}{| x- y|^2},
\end{align*}
and
\begin{align*}
\ds \mathbb{L}_1(x,y)&=\tau(x) h(x) \frac{\la \nu( x),\nu(y)\ra}{| x- y|^2}\\
\nm \ds &\q +2h(x)\frac{\la y-x,\nu(x) \ra \la \nu(x),\nu(y) \ra }{| x- y|^4}-\frac{\la h(x)\tau (x)\nu(x)+h^{'}(t)T(x),\nu(y) \ra }{| x- y|^2}\\
\nm \ds &\q - 2h(y)\frac{\la y-x,\nu(y) \ra \la \nu(x),\nu(y) \ra }{| x- y|^4}-\frac{\la h(y)\tau (y)\nu(y)+h^{'}(s)T(y),\nu(x) \ra }{| x- y|^2}.
\end{align*}
We get  from \eqref{H} and \eqref{L} that
\begin{align}\label{first-term-nuxnuy-Gamma}
\ds \Big[k {H_0^1}^{'}(k |\tilde x-\tilde y|)|\tilde x-\tilde y|\Big]\frac{\la \nu(\tilde x),\nu(\tilde y)\ra}{|\tilde x-\tilde y|^2}d\sigma(\tilde y)&=\Big[k {H_0^1}^{'}(k | x-y|)| x-y|\Big]\frac{\la \nu( x),\nu(y)\ra}{| x- y|^2}d\sigma( y)\nonumber\\
\nm \ds &\q +\sum_{n=1}^{\infty}\ep^n \sum_{m=0}^{n}\mathbb{H}_m(x,y) \mathbb{L}_{n-m}(x,y)d\sigma(y).
\end{align}

Using \eqref{n-tilde}, \eqref{tilde-x-y}, and \eqref{1-over-x-y}, we obtain
\begin{align*}
\ds \frac{\la \tilde y -\tilde x , \nu(\tilde x)\ra }{|\tilde x-\tilde y|^2}=&
\frac{\la y-x+\ep \big(h(y)\nu(y)-h(x)\nu(x)\big), \nu(x)-\ep \big(h(x)\tau (x)\nu(x)+h^{'}(t)T(x)\big)\ra}{| x- y|^2}\\
\nm \ds\q & \times  \frac{1}{{1+2\ep F(x,y)+\ep^2 G(x,y)}}\frac{1}{\sqrt{\big(1-\ep h(x) \tau(x)\big)^2+\ep^2 \big(h^{'}(t)\big)^2}}\\
\nm \ds:=& \sum_{n=0}^{\infty}\ep^n \mathbb{N}_{n}(x,y),
\end{align*}
where
\begin{align*}
\ds \mathbb{N}_{0}(x,y)= \frac{\la  y - x , \nu( x)\ra }{| x- y|^2},
\end{align*}
and
\begin{align}\label{N}
\ds \mathbb{N}_{1}(x,y)&=\tau(x)h(x)\frac{\la  y - x , \nu( x)\ra }{| x- y|^2}-\frac{\la y-x, h(x)\tau(x)\nu(x)+h^{'}(t) T(x) \ra}{| x- y|^2}\nonumber\\
\nm\ds&\q +h(y)\bigg( \frac{\la \nu(x),\nu(y)\ra}{| x- y|^2}-2 \frac{\la y-x, \nu(y)\ra\la y-x, \nu(x)\ra  }{| x- y|^4}\bigg)\nonumber\\
\nm \ds &\q+ h(x)\bigg( -\frac{1}{| x- y|^2}+2 \frac{\big (\la y-x, \nu(x)\ra\big )^2  }{| x- y|^4}\bigg).
\end{align}
By the  Taylor expansion and \eqref{x-y-square}, we get
\begin{align}\label{S}
k^2 {H_0^1}^{''}(k |\tilde y-\tilde x|)|\tilde y-\tilde x|^2:=\sum_{n=0}^{\infty}\mathbb{S}_{n}(x,y)=k^2 {H_0^1}^{''}(k |y- x|)| y- x|^2
+\sum_{n=1}^{\infty}\mathbb{S}_{n}(x,y),
\end{align}
with
$$
\mathbb{S}_{1}(x,y)=\Big[k^3 {H_0^1}^{'''}(k |y-x|)| y- x|+2 k^2 {H_0^1}^{''}(k |y- x|)\Big]
\la y-x, h(y)\nu(y)-h(x)\nu(x)\ra.
$$
Combining  \eqref{H}, \eqref{L}, \eqref{N}, and \eqref{S} yields the expansion
\begin{align}\label{second-term-nuxnuy-Gamma}
\ds &\Big[-k^2 {H_0^1}^{''}(k |\tilde x-\tilde y|)|\tilde x-\tilde y|^2+k {H_0^1}^{'}(k |\tilde x-\tilde y|)|\tilde x-\tilde y|\Big]\frac{\la \tilde y -\tilde x , \nu(\tilde x)\ra }{|\tilde x-\tilde y|^2}\frac{\la \tilde y -\tilde x , \nu(\tilde y)\ra }{|\tilde x-\tilde y|^2}
d\sigma(\tilde y)\nonumber \\
\nm \ds &=\Big[-k^2 {H_0^1}^{''}(k | x- y|)| x- y|^2+k {H_0^1}^{'}(k | x- y|)| x- y|\Big]\frac{\la  y - x , \nu( x)\ra }{| x-y|^2}\frac{\la  y - x , \nu( y)\ra }{| x- y|^2}
d\sigma( y)\nonumber\\
\nm \ds&\q +\sum_{n=1}^{\infty}\ep^n \sum_{m+p+q=1}^{n}\Big(\mathbb{H}_{m}(x,y)-\mathbb{S}_{m}(x,y\Big)
\mathbb{N}_{p}(x,y) \mathbb{M}_{q}(x,y)d\sigma(y).
\end{align}
Thanks  to \eqref{first-term-nuxnuy-Gamma} and \eqref{second-term-nuxnuy-Gamma}, we write
\begin{align}\label{nuxnuy}
\frac{\p^2 \Gamma_k(\tilde x-\tilde y)}{\p \nu(\tilde x)\p \nu(\tilde y)}d\sigma(\tilde y):=-\frac{i}{4}\sum_{n=0}^{\infty}\ep^n \mathbb{B}_{n}(x,y) d\sigma(y),
\end{align}
where
\begin{align*}
\ds -\frac{i}{4}\mathbb{B}_{0}(x,y)=\frac{\p^2 \Gamma_k( x- y)}{\p \nu( x)\p \nu( y)},
\end{align*}
and
\begin{align}\label{B1}
-\frac{i}{4}\mathbb{B}_{1}(x,y)&=h(x)\frac{\p^3 \Gamma_k( x- y)}{\p \nu^2( x)\p \nu( y)}-h^{'}(t)\frac{\p^2 \Gamma_k( x- y)}{\p T( x)\p \nu( y)}\nonumber \\
\nm \ds&\q + h(y)\frac{\p^3 \Gamma_k( x- y)}{\p \nu( x)\p \nu^2( y)}-\tau(y)h(y)\frac{\p^2 \Gamma_k( x- y)}{\p \nu( x)\p \nu( y)}-h^{'}(s)\frac{\p^2 \Gamma_k( x- y)}{\p \nu(x)\p T( y)}\nonumber\\
\nm \ds &=h(x)\frac{\p^3 \Gamma_k( x- y)}{\p \nu^2( x)\p \nu( y)}-h^{'}(t)\frac{\p^2 \Gamma_k( x- y)}{\p T( x)\p \nu( y)}\nonumber\\
\nm \ds &\q -\pd{}{\nu(x)}\frac{d}{ds}\Big(h(y)\frac{d \Gamma_k(x-y)}{ds} \Big)-k^2 h(y) \pd{\Gamma_k(x-y)}{\nu(x)}.
\end{align}
Introduce a sequence of integral operators $\big(\Acal_{D,n}^k\big)_{n\in \NN}$ defined for any $\phi \in L^2(\p D)$ by
$$
\Acal_{D,n}^k [\phi](x)=-\frac{i}{4} \int_{\p D}\mathbb{B}_{n}(x,y)\phi (y) d\sigma(y)\q \mbox{for } n\geq 0,
$$
with  $\Acal_{D,0}^k=\p(\Dcal_{D}^k)/\p \nu$. If $\phi \in \mathcal{C}^2(\p D)$, we get  from \eqref{jum-double-nu-squared} and  \eqref{B1} that
\begin{align}\label{A-1}
\ds  \Acal_{D,1}^{k}[\phi](x)=&\tau(x) h(x)\pd{(\Dcal_{D}^{k}[\phi])}{\nu}(x)-k^2 (\Kcal_{D}^{k})^*[ h\phi])(x)-k^2 h(x) \Kcal_{D}^k[\phi](x)\nonumber\\
\nm \ds &- (\Kcal_{D}^{k})^*\Big(\frac{d}{ds}\big(h \frac{d \phi}{ds}\big)\Big)(x)- \frac{d }{dt}\Big(h \frac{d (\Kcal_{D}^{k}[ \phi] ) }{dt}\Big)(x)\nonumber\\
\nm \ds =&\tau(x) h(x)\pd{(\Dcal_{D}^{k}[\phi])}{\nu}(x)-k^2\pd{(\Scal_{D}^{k}[ h\phi])}{\nu}\Big|_{\pm}(x)-k^2 h(x) \Dcal_{D}^k[\phi]\big|_{\pm}(x)\nonumber\\
\nm \ds &- \pd {\Scal^k_{D}}{\nu}\bigg(\frac{d}{ds}\Big(h \frac{d \phi}{ds}\Big)\bigg)\Big|_{\pm}(x)- \frac{d }{dt}\Big(h \frac{d (\Dcal_{D}^{k}[ \phi] ) }{dt}\Big)\Big|_{\pm}(x),\q x \in \p D.
\end{align}
 The operator $\Acal_{D,n}^k$ is bounded in $L^2(\p D)$ for $n\geq 1$. In fact,  it is an immediate consequence of the celebrate theorem of Coifman-MacIntosh-Meyer \cite{CMM}.

The results of the above asymptotic analysis is summarized in the following theorem.

\begin{thm} Let $N \in \NN$.
There exists  $C$ depending only on $k$, $\|X\|_{\mathcal{C}^2}$, and $\|h\|_{\mathcal{C}^1}$, such that for any ${\phi}_\ep \in L^2(\p D_\ep)$,  we have
\begin{align}\label{asymptotic-double-layer}
\bigg\|\Dcal_{D_\ep}^{k}[{\phi}_\ep]\circ \Psi_\ep\Big|_{\pm}-\Dcal_{D}^{k}[\phi]\Big|_{\pm}-\sum_{n=1}^{N}\ep^n \Dcal_{D,n}^{k}[\phi]\bigg\|_{L^2(\p D)}\leq C \ep^{N+1}\big\|\phi \big\|_{L^2(\p D)},
\end{align}
and
\begin{align}\label{asymptotic-normal-double-layer}
\bigg\|\pd{\Dcal_{D_\ep}^{k}[{\phi}_\ep]}{\nu}\circ \Psi_\ep-\pd{\Dcal_{D}^{k}[\phi]}{\nu}-\sum_{n=1}^{N}\ep^n \Acal_{D,n}^{k}[\phi]\Big|_{\pm}\bigg\|_{L^2(\p D)}\leq C \ep^{N+1}\big\|\phi \big\|_{L^2(\p D)},
\end{align}
where $\phi:=\phi_\ep\circ \Psi_\ep$.
\end{thm}

For  $\phi \in L^2(\p D)$. We introduce
\begin{align}\label{S-1}
\ds  \Scal_{D,1}^{k}[\phi](x) &=-\Scal_{D}^{k}[\tau h \phi](x)+h(\Kcal_{D}^{k})^*[\phi](x)+ \Kcal_{D}^{k}[h \phi](x)\nonumber\\
\nm \ds &=-\Scal_{D}^{k}[\tau h \phi](x)+\Big(h\pd{\Scal_{D}^{k}[ \phi]}{\nu}+\Dcal_{D}^{k}[ h \phi]\Big)(x) \Big |_{\pm},\q x\in \p D,
\end{align}
and
\begin{align}\label{K-1}
\ds  \Kcal_{D,1}^{k}[\phi](x)=&\tau(x) h(x)(\Kcal_{D}^{k})^*[ \phi](x)-\Kcal_{D}^{k}[\tau h \phi](x)\nonumber\\
\nm \ds &+ \pd{(\Dcal_{D}^{k}[ \phi])}{\nu}(x)- \frac{d }{ dt}\Big(h \frac{d(\Scal_{D}^{k}[ \phi] ) }{dt}\Big)(x)- k^2 h(x)\Scal_{D}^{k}[\phi](x),\nonumber\\
\nm \ds =&\Big(\tau h\pd{(\Scal_{D}^{k}[\phi])}{\nu}-\pd{(\Scal_{D}^{k}[\tau h\phi])}{\nu}\Big)\Big|_{\pm}(x)\nonumber\\
\nm \ds &+ \pd{(\Dcal_{D}^{k}[h \phi])}{\nu}(x)- \frac{d }{ dt}\Big(h \frac{d (\Scal_{D}^{k}[ \phi] ) }{dt}\Big)(x)- k^2 h(x)\Scal_{D}^{{k}}[\phi](x),\q x \in \p D.
\end{align}
It was proved in \cite{Z} that the operators  $ \Scal_{D,1}^{k}$  and  $ \Kcal_{D,1}^{k}$ are bounded in $L^2(\p D)$ and  the following proposition  holds.
\begin{prop}
There exists  $C$ depending only on $k$, $\|X\|_{\mathcal{C}^2}$, and $\|h\|_{\mathcal{C}^1}$, such that for any ${\phi}_\ep \in L^2(\p D_\ep)$,  we have
\begin{align}\label{asymptotic-single-layer}
\bigg\|\Scal_{D_\ep}^{k}[{\phi}_\ep]\circ \Psi_\ep-\Scal_{D}^{k}[\phi]-\ep \Scal_{D,1}^{k}[\phi]\bigg\|_{L^2(\p D)}\leq C \ep^2\big\|\phi \big\|_{L^2(\p D)},
\end{align}
and
\begin{align}\label{asymptotic-normal-single-layer}
\bigg\|\pd{\Scal_{D_\ep}^{k}[{\phi}_\ep]}{\nu}\circ \Psi_\ep\Big|_{\pm}-\pd{\Scal_{D}^{k}[\phi]}{\nu}\Big|_{\pm}-\ep \Kcal_{D,1}^{k}[\phi]\bigg\|_{L^2(\p D)}\leq C \ep^2\big\|\phi \big\|_{L^2(\p D)},
\end{align}
where $\phi:=\phi_\ep\circ \Psi_\ep$.
\end{prop}

\section{Proof of Theorem \ref{main-Theorem}}\label{proof-theorem1}
The solutions of \eqref{u} and \eqref{uep} are given by (see \cite{book-1,book-CK,book-Kirsch})
\begin{align}\label{representation-u}
\ds u^s(x)=\Scal^k_{D}\big(\pd{u^s}{\nu}\big)(x)-\Dcal^k_{D}\big(u^s\big)(x), \q x \mbox{ in }\RR^2\backslash \overline{D},
\end{align}
and
\begin{align}\label{representation-u-ep}
\ds u_\ep^s(x)=\Scal^k_{D_\ep}\big(\pd{u_\ep^s}{\nu}\big)(x)-\Dcal^k_{D_\ep}\big(u_\ep^s\big)(x), \q x \mbox{ in }\RR^2\backslash \overline{D_\ep}.
\end{align}

The following lemma holds.
\begin{lem}\label{estimes-Energy-Lemma}
 Let  $u^s$ and $u^s_\ep$ be the solutions of \eqref{u} and \eqref{uep}, respectively. For the case
 of a sound-soft obstacle, we suppose that $k^2$
is not an eigenvalue of $-\Delta$ on $D$ with Neumann boundary condition and $u^{in} \in \mathcal{C}^1(\p D)$, while for the case of a sound-hard obstacle, we suppose that $k^2$
is not an eigenvalue of $-\Delta$ on $D$ with Dirichlet boundary condition and $u^{in} \in \mathcal{C}^2(\p D)$.
 The following estimates hold:
\begin{align}\label{estimation-energy-1}
\ds \Big \|u^s_\ep \circ \Psi_\ep -u^s  \Big \|_{L^{2}(\p D)}\leq C \ep,
\end{align}
and
\begin{align}\label{estimation-energy-2}
\ds \bigg \|\pd{u^s_\ep}{\nu} \circ \Psi_\ep -\pd{u^s}{\nu}  \bigg \|_{L^2(\p D)}\leq C \ep,
\end{align}
with a constant $C$ independent of $\ep$.
\end{lem}

\begin{proof} \textbf{Sound-soft obstacle.}
 Let $x \in \p D$, then $\tilde x=\Psi_\ep(x)=x+\ep h(x)\nu(x) \in \p D_\ep$.  We have
\begin{align*}
\ds u^s_\ep(\tilde x)- u^s(x)=u^{in}( x)- u^{in}\big(x+\ep h(x)\nu(x)\big),
\end{align*}
from which it follows by using  the mean value theorem that $\|u^s_\ep \circ \Psi_\ep -u^s   \|_{L^{\infty}(\p D)}\leq C \ep$. Then,
one can see   from the injection continuous $L^{\infty}(\p D)\hookrightarrow L^2(\p D)$ that
\eqref{estimation-energy-1} is true.

It follows from \eqref{representation-u}, \eqref{representation-u-ep}, and  the jump formula \eqref{jump-normal-single-layer} that
\begin{align}\label{jump-estimates}
\ds \Big(-\frac{1}{2}I+ (\Kcal_{D}^k)^{*}\Big)\big(\pd{u^s}{\nu}\big)( x)
=\pd{\Dcal^k_{D}(u^s)}{\nu}( x), \q x \in \p D,
\end{align}
and
\begin{align*}
\ds \Big(-\frac{1}{2}I+ (\Kcal_{D_\ep}^k)^{*}\Big)\big(\pd{u^s_\ep}{\nu}\big)(\tilde x)
=\pd{\Dcal_{D_\ep}^k(u^s_\ep)}{\nu}(\tilde x), \q \tilde x \in \p D_\ep.
\end{align*}
%the previous integral  equation has the  following expansion
The following  expansion follows from \eqref{asymptotic-normal-single-layer}, \eqref{asymptotic-normal-double-layer}, and the above equation 
\begin{align}\label{integral-equation-asymptotic}
\ds \Big(-\frac{1}{2}I +(\Kcal_{D}^k)^{*}\Big)\big(\pd{u^s_\ep}{\nu}\circ \Psi_\ep \big)( x)
=\pd{\Dcal^k_{D}(u^s_\ep\circ \Psi_\ep)}{\nu}( x)+O(\ep), \q x \in \p D.
\end{align}
Subtracting  \eqref{jump-estimates} from  \eqref{integral-equation-asymptotic} yields
\begin{align*}
\ds \Big(-\frac{1}{2}I+ (\Kcal_{D}^k)^{*}\Big)\big(\pd{u^s_\ep}{\nu}\circ \Psi_\ep-\pd{u}{\nu}\big)( x)
=\pd{\Dcal^k_{D}(u^s_\ep\circ \Psi_\ep-u)}{\nu}( x)+O(\ep), \q x \in \p D.
\end{align*}
Using the fact   that $-({1}/{2})I+(\Kcal_{D}^k)^{*}$ is invertible on $L^2(\p D)$ and \eqref{estimation-energy-1} to deduce \eqref{estimation-energy-2}.

\textbf{Sound-hard obstacle.} For $\tilde x=x+\ep h(x) \nu(x) \in \p D_\ep$. We have
\begin{align*}
\ds \pd{ u^s_\ep}{\nu}(\tilde x)- \pd{u^s}{\nu}(x)=\pd{u^{in}}{\nu}( x)- \pd{u^{in}}{\nu}(\tilde x).
\end{align*}
Since, by using  the mean value theorem and
 the injection continuous $L^{\infty}(\p D)\hookrightarrow L^2(\p D)$, we get
\eqref{estimation-energy-2}. It follows from  \eqref{representation-u}, \eqref{representation-u-ep}, and the jump formula \eqref{jump-double-layer} that
\begin{align}\label{jump-estimates-hard}
\ds \Big(\frac{1}{2}I+ \Kcal_{D}^k\Big)\big(u^s\big)( x)
=\Scal^k_{D}\big(\pd{u^s}{\nu}\big)( x),\q x \in \p D,
\end{align}
and
\begin{align}\label{above}
\ds \Big(\frac{1}{2}I+ \Kcal_{D_\ep}^k\Big)(u^s_\ep)(\tilde x)
=\Scal^k_{D_\ep}\big(\pd{u^s_\ep}{\nu}\big)(\tilde x), \q \tilde x \in \p D_{\ep}.
\end{align}
According to  \eqref{asymptotic-normal-double-layer},  \eqref{asymptotic-normal-single-layer},  and \eqref{above},
the following asymptotic expansion holds
\begin{align}\label{integral-equation-asymptotic-hard}
\ds \Big(\frac{1}{2}I +\Kcal_{D}^k \Big)\big(u^s_\ep\circ \Psi_\ep\big)( x)
=\Scal^k_{D}(\pd{u^s_\ep}{\nu}\circ \Psi_\ep)( x)+O(\ep), \q x \in \p D.
\end{align}
From  \eqref{jump-estimates-hard} and \eqref{integral-equation-asymptotic-hard}, we get
\begin{align*}
\ds \Big(\frac{1}{2}I+ \Kcal_{D}^k\Big)\big(u^s_\ep\circ \Psi_\ep-u^s\big)( x)
=\Scal^k_{D}\big(\pd{u^s_\ep}{\nu}\circ \Psi_\ep-\pd{u^s}{\nu}\big)( x)+O(\ep), \q x \in \p D.
\end{align*}
Clearly the estimate \eqref{estimation-energy-1} immediately follows from \eqref{estimation-energy-2} and the fact   that $({1}/{2})I+\Kcal_{D}^k$ is invertible on $L^2(\p D)$. Thus   the proof of Lemma  \ref{estimes-Energy-Lemma} is complete.
\end{proof}

Now  we are ready to prove Theorem \ref{main-Theorem}. Let $v^s$ be the solution of \eqref{v}. It then follows from \eqref{main-equality} that
\begin{align*}
\ds &\int_{\p D }\Big (\pd{\Scal_D^k[\phi]}{\nu}\Big|_{+}v^s-\Scal^k_{D}[\phi]\pd{v^s}{\nu}\Big)d\sigma=
 \int_{\p D }\Big (\pd{\Dcal_D^k[\psi]}{\nu}v^s-\Dcal^k_{D}[\psi]\big|_{+}\pd{v^s}{\nu}\Big)d\sigma=0,
\end{align*}
and
\begin{align*}
\ds &\int_{\p D }\Big (\Kcal^k_{D,1}[\phi]-\Acal^k_{D,1}[\psi]\Big)v^sd\sigma-
 \int_{\p D }\Big (\Scal^k_{D,1}[\phi]-\Dcal^k_{D,1}[\psi]\Big)\pd{v^s}{\nu}d\sigma\\
\nm \ds & =-\int_{\p D}h\Big(\pd{\Scal_{D}^k[\phi]}{\nu}\Big|_{+}-\pd{\Dcal_{D}^k[\psi]}{\nu}\Big)\pd{v^s}{\nu} d\sigma\\
\nm \ds &\q+\int_{\p D}\bigg[ \tau h \Big(\pd{\Scal_{D}^k[\phi]}{\nu}\Big|_{+}-\pd{\Dcal_{D}^k[\psi]}{\nu}\Big)
- \frac{d}{dt}\bigg(h \frac{d}{dt}\Big(\Scal_D^k[\phi]-\Dcal_D^k[\psi]\big|_{+}\Big)\bigg)\\
\nm \ds &\q\q\q\q\q\q\q\q\q\q\q\q\q\q\q\q\q\q\q-h k^2 \Big(\Scal_D^k[\phi]-\Dcal_D^k[\psi]\big|_{+}\Big)\bigg] v^s d\sigma.
%\nm\ds &\q +\int_{\p D}\Big(\Scal^k_{D}[\tau h \phi]\pd{v^s}{\nu}-\pd{\Scal^k_{D}[\tau h \phi]}{\nu}v^s
%\Big) d\sigma +\int_{\p D}\Big(\pd{\Dcal^k_{D}[ h \phi]}{\nu}v^s-\Dcal^k_{D}[ h \phi]\pd{v^s}{\nu}
%\Big) d\sigma\\
%\nm\ds &\q
\end{align*}
Put $\ds \phi=\pd{u^s_\ep}{\nu}\circ \Psi_\ep$ and $\ds \psi={u^s_\ep}\circ \Psi_\ep$.  It follows from \eqref{estimation-energy-1} and \eqref{estimation-energy-2} that
\begin{align*}
\Scal_D^k[\phi]-\Dcal_D^k[\psi]\big|_{+}=u^s+O(\ep) \q\mbox{and}\q \pd{\Scal_{D}^k[\phi]}{\nu}\Big|_{+}-\pd{\Dcal_{D}^k[\psi]}{\nu}=\pd{u^s}{\nu}+O(\ep)\q  \mbox{on } \p D,
\end{align*}
and then the asymptotic expansions  in Theorem \ref{main-Theorem} of $
[u^s_\ep, v^s,\Psi_\ep, D]
$
are proved  as desired.

\section{Asymptotic expansions for the DNO and NDO}
For a given bounded  domain $D$ with $\mathcal{C}^2$-boundary. We introduce the DNO for the exterior Helmholtz problem   which is   defined by
%\begin{align*}
%\ds \mathcal{N}: L^2 (\p D)\rightarrow H^{-1}(\p D)
%\end{align*}

\begin{align*}
\ds \mathcal{N}_0(f)=\pd{u}{\nu}\big|_{\p D},
\end{align*}
where $u$ is the solution to
\begin{align}\label{u-DtN}
 \left\{ \begin{array}{l}
\ds \Delta u +k^2 u=0 \quad\mbox{in }\RR^2\backslash \overline{D},\\
\nm\ds\bigg|\pd{u}{r}-ik u \bigg|=O\Big(1/r^{\frac{3}{2}}\Big)\q \mbox{as } r=|x|\rightarrow +\infty \q \mbox{uniformy in }\frac{x}{|x|},\\
\nm\ds u(x)=f(x) \quad \mbox{for } x\in  \p D.\end{array}\right.
\end{align}
 Let $\mathcal{N}_{\ep}[f]$ be the perturbed DNO resulting from small perturbations of $D$, namely,
$$
\mathcal{N_\ep}(f)(x)=\pd{u_\ep}{\nu}\circ \Psi_\ep(x),\q\q \Psi_\ep(x)=x+\ep h(x) \nu(x)\q \mbox{for }
x\in \p D,
$$
where
\begin{align}\label{u-DtN-ep}
 \left\{ \begin{array}{l}
\ds \Delta u_\ep +k^2 u_\ep=0 \quad\mbox{in }\RR^2\backslash \overline{D_\ep},\\
\nm\ds\bigg|\pd{u_\ep}{r}-ik u_{\ep} \bigg|=O\Big(1/r^{\frac{3}{2}}\Big)\q \mbox{as } r=|x|\rightarrow +\infty \q \mbox{uniformy in }\frac{x}{|x|},\\
\nm\ds u_\ep\circ \Psi_\ep(x)=f(x) \quad \mbox{for } x\in  \p D.\end{array}\right.
\end{align}
In connection with the results for rough non-periodic surfaces \cite{Coifman,Milder} and periodic interfaces \cite{LLZ}. The following theorem holds:
\begin{thm} \label{theorem-D-t-N} Suppose that $k^2$
is not an eigenvalue of $-\Delta$ on $D$ with Neumann boundary condition
and $f \in \mathcal{C}^2(\p D)$.
The following expansion  holds:
$$
\mathcal{N}_{\ep }(f)(x)=\mathcal{N}_{0 }(f)(x)+ \ep \Big(-\frac{1}{2}I+(\Kcal_{D}^k\big)^*\Big)^{-1}\bigg(\Acal^k_{D,1}[f]-\Kcal^k_{D,1}\big[\mathcal{N}_{0 }(f)\big]
\bigg)(x)+O(\ep^2),
$$
where the operators $\Acal^k_{D,1}$  and $\Kcal^k_{D,1}$ are defined in \eqref{A-1} and \eqref{K-1}, respectively. Here the remainder $O(\ep^2)$ depends  only
on the $\mathcal{C}^2$-norm of $X$, the $\mathcal{C}^1$-norm of $h$,  and $k$.
\end{thm}
\begin{proof} Let $u_\ep$ be the solution  to \eqref{u-DtN-ep}. Then the following representation formula holds
\begin{align*}
\ds u_\ep(x)=\Scal_{D_{\ep}}^{k}\big[\pd{u_\ep}{\nu}\big](x)- \Dcal_{D_{\ep}}^k[u_\ep](x), \q  x \in \RR^2\backslash \overline{D_\ep}.
\end{align*}
Therefore the jump formula \eqref{jump-normal-single-layer} yields
\begin{align*}
\pd{u_\ep }{\nu}\circ \Psi_\ep(x)=\Big(\frac{1}{2}I+\big(\Kcal_{D_\ep}^k\big)^*\Big)\big[\pd{u_\ep}{\nu}\big]\circ \Psi_\ep(x)-\pd{\Dcal_{D_{\ep}}^k[u_\ep]}{\nu}\circ \Psi_\ep(x), \q x \in \p D.
\end{align*}
It then follows from  \eqref{asymptotic-normal-double-layer} and \eqref{asymptotic-normal-single-layer} that
\begin{align}\label{01}
\ds \Big(-\frac{1}{2}I+\big(\Kcal_{D}^k\big)^*\Big)\big[\mathcal{N}_{\ep }(f)
\big]=\pd{\Dcal_{D}^k[f]}{\nu}+\ep \bigg(\Acal^k_{D,1}[f]-\Kcal^k_{D,1}\big[\mathcal{N}_{\ep }(f)\big]
\bigg)+O(\ep^2)\q \mbox{on } \p D.
\end{align}
Similarly, one can checks that 
\begin{align}\label{02}
\ds \Big(-\frac{1}{2}I+\big(\Kcal_{D}^k\big)^*\Big)\big[\mathcal{N}_{0 }(f)
\big]=\pd{\Dcal_{D}^k[f]}{\nu}\q \mbox{on } \p D.
\end{align}
Subtraction \eqref{02} from \eqref{01} yields
\begin{align*}
\ds \Big(-\frac{1}{2}I+\big(\Kcal_{D}^k\big)^*\Big)\big[\mathcal{N}_{\ep }(f)-\mathcal{N}_{0 }(f)
\big]=\ep \bigg(\Acal^k_{D,1}[f]-\Kcal^k_{D,1}\big[\mathcal{N}_{\ep }(f)\big]
\bigg)+O(\ep^2)\q \mbox{on } \p D.
\end{align*}
If $k^2$
is not an eigenvalue of $-\Delta$ on $D$ with Neumann boundary condition, then we have from Lemma \ref{lemma-invertible} that $-({1}/{2})I+\big(\Kcal_{D}^k\big)^*$ is invertible on $L^2(\p D)$. Hence
\begin{align*}
\ds \mathcal{N}_{\ep }(f)-\mathcal{N}_{0 }(f)
=\ep \Big(-\frac{1}{2}I+\big(\Kcal_{D}^k\big)^*\Big)^{-1}\bigg(\Acal^k_{D,1}[f]-\Kcal^k_{D,1}\big[\mathcal{N}_{\ep }(f)\big]\bigg)+O(\ep^2)\q \mbox{on } \p D.
\end{align*}
Note That
\begin{align}\label{estimate-DtN}
\ds \big\|\mathcal{N}_{\ep }(f)-\mathcal{N}_{0 }(f)\big \|_{L^2(\p D)}\leq C\ep.
\end{align}
This completes the proof.
\end{proof}

Now,  let us  introduce the NtD  operator for the exterior Helmholtz problem
%$$
%\Lambda: L^2 (\p D)\rightarrow H^{1}(\p D),
%$$
which is defined by
$$
\Lambda_0[g]=v \big|_{\p D},
$$
where $v$ is the solution to
\begin{align}\label{v-0}
 \left\{ \begin{array}{l}
\ds \Delta v +k^2 v=0 \quad\mbox{in }\RR^2\backslash \overline{D},\\
\nm\ds\bigg|\pd{v}{r}-ik v \bigg|=O\Big(1/r^{\frac{3}{2}}\Big)\q \mbox{as } r=|x|\rightarrow +\infty \q \mbox{uniformy in }\frac{x}{|x|},\\
\nm\ds \pd{v}{\nu}(x)=g(x) \quad \mbox{for } x\in  \p D.\end{array}\right.
\end{align}
We let $\Lambda_{\ep}[g]$ be the perturbed NtD operator caused par $D_\ep$, that is,
$$
\Lambda_{\ep}[g](x)=v_\ep\circ \Psi(x),
$$
where
\begin{align}\label{v-N-t-D}
 \left\{ \begin{array}{l}
\ds \Delta v_\ep +k^2 v_\ep=0 \quad\mbox{in }\RR^2\backslash \overline{D_\ep},\\
\nm\ds\bigg|\pd{v_\ep}{r}-ik v_{\ep} \bigg|=O\Big(1/r^{\frac{3}{2}}\Big)\q \mbox{as } r=|x|\rightarrow +\infty \q \mbox{uniformy in }\frac{x}{|x|},\\
\nm\ds \pd{v_\ep}{\nu}\big(x+\ep h(x)\nu(x)\big)=g(x) \quad \mbox{for } x\in  \p D,\end{array}\right.
\end{align}
The following theorem holds.
\begin{thm} Suppose that $k^2$
is not an eigenvalue of $-\Delta$ on $D$ with Dirichlet boundary condition
and $g\in \mathcal{C}^2(\p D)$.
The following asymptotic formula holds:
$$
{\Lambda}_{\ep }[g](x)=\Lambda_{0 }[g](x)+ \ep \Big(\frac{1}{2}I+\Kcal_{D}^k\Big)^{-1}\bigg(\Scal^k_{D,1}[g]-\Dcal^k_{D,1}\Big(\Lambda_{0 }[g]\Big)\bigg)(x)+O(\ep^2),
$$
where the operators $\Dcal^k_{D,1}$  and $\Scal^k_{D,1}$ are defined in \eqref{D-1} and \eqref{S-1}, respectively. Here  the remainder $O(\ep^2)$ depends  only
on the $\mathcal{C}^2$-norm of $X$, the $\mathcal{C}^1$-norm of $h$,  and $k$.
\end{thm}
\begin{proof} The  solution of \eqref{v-N-t-D} is given by
\begin{align}\label{representation-v-ep}
\ds v_\ep(x)=\Scal_{D_{\ep}}^{k}\big[\pd{v_\ep}{\nu}\big](x)- \Dcal_{D_{\ep}}^k[v_\ep](x), \q  x \in \RR^2\backslash \overline{D_\ep}.
\end{align}
From the jump formula \eqref{jump-double-layer} and \eqref{representation-v-ep} we deduce
\begin{align*}
v_\ep \circ \Psi_\ep(x)=\Scal_{D_{\ep}}^{k}\big[\pd{v_\ep}{\nu}\big]\circ\Psi_\ep (x)-\big(-\frac{1}{2}I+\Kcal_{D_\ep}^k\big)\big[v_\ep\big]\circ \Psi_\ep(x), \q x \in \p D.
\end{align*}
It then follows from  \eqref{asymptotic-double-layer} and \eqref{asymptotic-single-layer} that
\begin{align}\label{01}
\ds \Big(\frac{1}{2}I+\Kcal_{D}^k\Big)\big[\Lambda_{\ep }(g)
]=\Scal_{D}^k[g]+\ep \Big(\Scal^k_{D,1}[g]
-\Dcal^k_{D,1}\big[\Lambda_{\ep }(g)\big]\Big)+O(\ep^2)\q \mbox{on } \p D.
\end{align}
In the same way as above, we can get
\begin{align}\label{02}
\ds  \Big(\frac{1}{2}I+\Kcal_{D}^k\Big)\big[\Lambda_{0}(g)
]=\Scal_{D}^k[g]\q \mbox{on } \p D.
\end{align}
Subtraction \eqref{02} from \eqref{01} yields
\begin{align*}
\ds \Big(\frac{1}{2}I+\Kcal_{D}^k\Big)\big[\Lambda_{\ep }(g)-\Lambda_{0 }(g)
\big]=\ep \Big(\Scal^k_{D,1}[g]
-\Dcal^k_{D,1}\big[\Lambda_{\ep }(g)\big]\Big)+O(\ep^2)\q \mbox{on } \p D.
\end{align*}
If $k^2$
is not an eigenvalue of $-\Delta$ on $D$ with Dirichlet  boundary condition, then we have from Lemma \ref{lemma-invertible} that $({1}/{2})I+\Kcal_{D}^k$ is invertible on $L^2(\p D)$. Hence
\begin{align*}
\ds\Lambda_{\ep }(g)-\Lambda_{0 }(g)=\ep \Big(\frac{1}{2}I+\Kcal_{D}^k\Big)^{-1}\Big(\Scal^k_{D,1}[g]
-\Dcal^k_{D,1}\big[\Lambda_{\ep }(g)\big]\Big)+O(\ep^2)\q \mbox{on } \p D.
\end{align*}
Since
\begin{align}\label{estimate-NtD}
\ds \big\|\Lambda_{\ep }(g)-\Lambda_{0 }(g)\big \|_{L^2(\p D)}\leq C\ep,
\end{align}
and the theorem is proved.
\end{proof}
%\section{Reconstructing formulas}

Based on the same arguments given in the proofs of Theorem \ref{main-Theorem}, the following theorem holds.
\begin{thm} \label{main-theorem-reconstructing} Let $f,g \in \mathcal{C}^2(\p D)$. The following reconstructing formulas hold:

\begin{align}
\ds &\int_{\p D}\Big(\mathcal{N}_\ep(f) g-f \mathcal{N}_0(g)\Big) d\sigma \nonumber\\
\nm \ds &\q=\ep \int_{\p D}h\bigg(\frac{\p f}{\p T}\frac{\p g}{\p T}+(n-1)\tau  \mathcal{N}_0(f) g
-\mathcal{N}_0(f)\mathcal{N}_0(g)-k^2 f g\bigg)d\sigma +O(\ep^2),\label{main-Theorem-DtN}
\end{align}
and
\begin{align}
\ds &\int_{\p D}\Big(f \Lambda_{0}(g)-\Lambda_{\ep}(f) g\Big) d\sigma\nonumber\\
\nm \ds &\q =\ep \int_{\p D}h\bigg(\frac{\p \Lambda_{0}(f)}{\p T}\frac{\p \Lambda_{0}(g)}{\p T}+(n-1)\tau  f \Lambda_{0}(g)
-f g -k^2\Lambda_{0}(f)\Lambda_{0}(g)\bigg)d\sigma
+O(\ep^2), \label{main-Theorem-NtD}
\end{align}
where the remainder $O(\ep^2)$ depends  only
on the $\mathcal{C}^2$-norm of $X$, the $\mathcal{C}^1$-norm of $h$,  and $k$.
\end{thm}

\section{Reconstruction of the shape deformation}
Formulas in   \eqref{main-Theorem-equality}, \eqref{main-Theorem-DtN}, and \eqref{main-Theorem-NtD}
can be used to reconstruct an approximation of the deformation $h$  by choosing test functions of the integral in the right-hand side appropriately. Let us treat the formulas in \eqref{main-Theorem-equality}. The reconstruction of the shape deformation  from \eqref{main-Theorem-DtN} and \eqref{main-Theorem-NtD} can be done in the same way.

To illustrate this, let us consider $D$ to be the disk centred at the origin with radius $\rho$. For an integer $n$ set
\begin{align*}
\ds u_n(r,\theta)=H^{(1)}_{|n|}(kr) e^{in\theta}\q \mbox{for } r>\rho.
\end{align*}
Since $u_n$  satisfies  $(\Delta +k^2 )u_n=0$ in $\RR^2\backslash \overline{D}$ and the summerfield condition
$$\
\Big|\frac{\p u_{n}}{\p r}-ik u_{n}\Big|=O(1/r^{\frac{3}{2}})\q \mbox{as } r\rightarrow \infty.
$$
We then take $u^s=u_n$ and $v^s=u_m$ in
%formula
\eqref{main-Theorem-equality} to get
% with $u^i=-u_n$ and $v^i=-u_m$. It then follows from \eqref{main-Theorem-equality} that
\begin{align}\label{main-reconstruction-shape}
\ds [u^s_\ep,v^s,\Psi, D]=&\ep c_{n,m}(\rho, k) \int_{\p D}h(\theta) e^{i(n+m)\theta}d\theta+O(\ep^2),
\end{align}
with  
\begin{align*}
\ds c_{n,m}(\rho, k)=\Big[-nm+\tau k \sigma_1(\rho, n,k)+k^2\sigma_1(\rho, n, k)\sigma_1(\rho, m,k)-k^2\Big]
|H^{(1)}_{|n|}(k\rho)|H^{(1)}_{|m|}(k\rho),
\end{align*}
%where the so-called discrete symbol $\sigma_1$ is given by
where   $\sigma_1$ is given by
\begin{align*}
\ds \sigma_1(\rho, n,k)=k \frac{H^{(1)^{'}}_{|n|}(k\rho)}{H^{(1)}_{|n|}(k\rho)}
=-k \frac{ H^{(1)}_{|n|+1}(k\rho)}{H^{(1)}_{|n|}(k\rho)}+|n|.
\end{align*}
%Thus, for fixed $k$, we have
%\begin{align*}
%\ds \sigma_1(n,\rho, k)\sim |n|\q \mbox{as } |n|\rightarrow \infty.
%\end{align*}
Formulas in  \eqref{main-Theorem-equality} 
show that the Fourier coefficients $h_p$
of $h$ can be determined from measurements
on  $\p D_\ep$
by varying the test function $v^s$, provided
that the order of magnitude of  $|h_p|$ is much larger than $\ep$.

If $D$ is a ball of   radius $\rho$. Hence $h$ can be expanded as
\begin{align}\label{coeff-hlm}
\ds h(\theta, \varphi)=\sum_{l=0}^{\infty}\sum_{m=-l}^l h_{l}^m Y_{l}^m(\theta, \varphi),
\end{align}
where  $Y_{l}^m$, for $l\geq 0$ and $-l\leq m\leq l$ are the spherical harmonics of order $l$. These functions  constitute an orthogonal basis of the space linear of $L^2(\p D)$ and satisfy  $\overline{Y_{l}^m}=(-1)^mY_{l}^{-m}$ (see \cite {book-Nedelec, book-Kirsch}).  The coefficients $ h_{l}^m$  in \eqref{coeff-hlm} are defined  by
\begin{align*}
\ds h_{l}^m=\int_{\p D} h(\theta, \varphi)\overline{ Y_{l}^{m}}(\theta, \varphi)d\sigma.
\end{align*}
For two integers $l$ and $m$ set
\begin{align*}
\ds u_{l,m}(r,\theta)=h^{(1)}_{l}(kr) Y_{l}^{m}(\theta, \varphi)\q \mbox{for } r>\rho,
\end{align*}
where $h^{(1)}_{l}$ is the spherical Hankel function of the first kind of order $l$.
Since  $u_{l,m}$  satisfies
\begin{align*}
(\Delta +k^2 )u_{l,m}=0\q \mbox{in }\RR^3\backslash \overline{D}, \q\q
\Big|\frac{\p u_{l,m}}{\p r}-ik u_{l,m}\Big|=O(1/r^{2})\q \mbox{as } r\rightarrow \infty.
\end{align*}
Set $u^s=u_{0,0}$ and $v^s=u_{l,m}$.  One can easily check  that
$$
\frac{\p u^{s}}{\p T}\frac{\p v^{s}}{\p T}+2\tau \pd{u^{s}}{\nu} v^{s}
-\pd{u^{s}}{\nu} \pd{v^{s}}{\nu}-k^2 u^s v^s=c_{l,m}( \rho, k) Y_{l}^{m}(\theta, \varphi) \q \mbox{on } \p D.
$$
Then  by measuring $[u^s_\ep, v^s , \Psi_\ep, D]$  in \eqref{main-Theorem-equality}, we can reconstruct $ h_{l}^{-m}$. This implies that the coefficients $h_l^m$ of $h$ can be determined by varying the test function $v^s=u_{l,m}$, provided
that the order of  magnitude of  $|h_l^m|$ is much larger than $\ep$.

\end{document}